%% file: min-infocom.tex
\documentclass[10pt,twocolumn,twoside]{infocom}
\usepackage{amsmath}
\usepackage{amssymb,mathptm}
\usepackage{graphicx}
\newtheorem{theo}{Theorem}
\newtheorem{remark}{Remark}

\makeatletter
\@ifundefined{proof}{
\def\QED{\mbox{\rule[0pt]{1.5ex}{1.5ex}}}

}{}
\@ifundefined{keywords}{

}{}

\makeatother

\newcommand\egaldef{\stackrel{\mbox{\upshape\tiny def}}{=}}
\DeclareMathOperator{\PP}{I\!P}
\DeclareMathOperator{\EE}{I\!E}
\newcommand{\NN}{{\mathrm{I\! N}}}
\newcommand{\1}{\leavevmode\hbox{\textrm{\small1\kern-0.35em\normalsize1}}}
\newcommand{\ind}[1]{\1_{\{#1\}}}

\newcommand{\N}{^{\scriptscriptstyle (N)}}
\newcommand{\Rfair}{^{\text{\upshape mm}}}
\newcommand{\Rmin}{^{\text{\upshape min}}}
\newcommand{\inb}{^{\mbox{\tiny in}}}
\newcommand{\outb}{^{\mbox{\tiny out}}}

\newcommand{\EX}{{\bar \alpha}}
\newcommand{\R}{{\mathcal{R}}}
\renewcommand{\L}{{\mathcal{L}}}

\begin{document}

\title{Best-effort networks: modeling and performance analysis via
       large networks asymptotics}
\author{Guy Fayolle,  Arnaud de La Fortelle, Jean-Marc Lasgouttes,
        Laurent Massouli\'e, James Roberts
\thanks{This work has been partly supported by a grant from
        France T\'el\'ecom R\&D.}
\thanks{G. Fayolle, A. de La Fortelle and J.-M. Lasgouttes are with
       INRIA, Domaine de Voluceau BP 105, Rocquencourt 78153 Le
       Chesnay CEDEX, France.}
\thanks{L. Massouli\'e is with Microsoft Research, Saint
       George House, 1 Guildhall Street, CB2 3NH Cambridge,
       United Kingdom.}
\thanks{J. Roberts is with France T\'el\'ecom R\&D, 38--40, rue du 
       G\'en\'eral Leclerc, 92794 Issy les Moulineaux CEDEX 9,
       France.}
}

\markboth{IEEE Infocom 2001}
         {Fayolle \emph{et al}: Best-effort networks: modeling and
          performance analysis via large networks asymptotics}

\maketitle

\begin{abstract}
In this paper we introduce a class of Markov models, termed
best-effort networks, designed to capture performance indices such as
mean transfer times in data networks with best-effort service. We
introduce the so-called min bandwidth sharing policy as a conservative
approximation to the classical max-min policy. We establish necessary
and sufficient ergodicity conditions for best-effort networks under
the min policy. We then resort to the mean field technique of
statistical physics to analyze network performance deriving fixed
point equations for the stationary distribution of large symmetrical
best-effort networks. A specific instance of such networks is the
star-shaped network which constitutes a plausible model of a network
with an overprovisioned backbone. Numerical and analytical study of
the equations allows us to state a number of qualitative conclusions
on the impact of traffic parameters (link loads) and topology
parameters (route lengths) on mean document transfer time.
\end{abstract}

\begin{keywords}
best-effort service, max-min fairness, min policy, mean field,
star-shaped network. 
\end{keywords}

\input{inf-intro}

\input{inf-ergo}

\input{inf-asympt}

\input{inf-numer}

\bibliography{min-infocom}
\bibliographystyle{IEEE}

\end{document}

%% file: inf-intro.tex
\section{Introduction}

Consider a network handling data flows from several users, and
assume no quality of service commitments (such as minimum bandwidth
allocations) have been made by the network to the users. Such a
situation has been prevalent in the Internet until now, and is likely
to remain so for another few years.

The preferred service model in that situation, known as best effort
service, consists in allocating a fair proportion of bandwidth to
contending users; see, e.g., Bertsekas and Gallager~\cite{BerGal:1}.
There are actually several possible notions of fairness available for
this bandwidth allocation problem (see, e.g., Mo and
Walrand~\cite{mowal} for a parametric family of fairness criteria
covering all other notions proposed so far), although the classical
notion proposed in~\cite{BerGal:1} is the so-called max-min fairness.

Recent work has led to a relatively good understanding of how bandwidth is shared between network users when a given congestion control algorithm
is used; see, e.g., Massouli\'e and Roberts~\cite{mass} and references
therein. The question of what type of fairness is achieved in the
current Internet, where Jacobson's congestion avoidance algorithm---as
implemented in TCP---is responsible for congestion control, has been
studied in depth by Hurley et al.~\cite{hurl}. These studies all assume 
the number of flows remains fixed.
 
In comparison, there is little work accounting for the random nature
of traffic and its impact on user perceived quality of service.
Consider for instance the transfer of digital documents (Web pages,
files, emails,\ldots) using a transport protocol like TCP\@. This
constitutes the bulk of Internet traffic today. The performance
criterion relevant to such transfers is the overall document transfer
time. This time is clearly highly dependent on the number of ongoing
transfers on the links shared by the considered connection. This
number varies as a random process as new connections are established
and existing ones terminate in a way which depends on how bandwidth is
allocated, as well as on the underlying traffic parameters.

In the case of a single bottleneck resource shared perfectly fairly,
simple traffic assumptions of Poisson arrivals and identically and
independently distributed document size lead to a processor sharing
queueing model~\cite{mass2}. This fluid flow model provides useful
results on expected response times as a function of the load of an
access link or a Web server, for instance. It also shows how a form of
congestion collapse can occur when demand (arrival rate $\times$ mean
document size) exceeds capacity. The processor sharing queue is then
no longer ergodic leading to unbounded response times. To understand
the impact of multiple bottlenecks and to investigate the effect of
different sharing strategies, one would like to dispose of similar
analytical results for multiple resource systems.

To the best of our knowledge, the only analytical results available so
far are in Massouli\'e and Roberts~\cite{mass2}, where the so-called
linear network topology is investigated. Simulation results for the
linear network can be found both in~\cite{mass2} and in de Veciana et
al.~\cite{konst}. The main motivation for the present paper is to
study the performance of best-effort networks with alternative
topologies, the ultimate objective being the derivation of heuristics 
enabling the performance evaluation of bandwidth sharing in a general 
network. 

In the present paper we report the results of our preliminary 
investigations. These include an analysis of the stability conditions 
under which the expected response time remains finite in a general 
network. We also apply mean field techniques to evaluate the performance 
of large symmetrical networks. Numerical results derived from the model 
illustrate how response times depend on the number of bottleneck links 
and their utilization. These results are of some practical interest and 
aide our understanding of the behavior of best effort networks. A 
further significant contribution is the insight provided into the 
inherent difficulty of deriving performance estimates when more than one 
bottleneck limits throughput.  

Section~\ref{sec.best} introduces a general class of Markov models for
best-effort networks which is intended to capture the impact of
network topology, traffic parameters and bandwidth sharing (fairness)
criteria on document transfer times. A brief account of the results
obtained in~\cite{mass2} is given, and the so-called ``min'' bandwidth
allocation is introduced as a conservative approximation to max-min
fairness. Section~\ref{sec.ergo} then establishes the necessary and
sufficient ergodicity criteria for best-effort networks under the
``min'' policy. Section~\ref{sec.asympt} introduces the so-called
``star topology''. Its relevance as a model of real networks is
discussed and a mean field heuristic is proposed. This heuristic is
expected to be accurate in the asymptotic regime where the number of
star branches is large. The derived fixed point equations are
investigated numerically in Section~\ref{sec.numer}. Simulation is
used to verify the accuracy of the heuristics. Extensions to the
star-shaped network are also considered in Sections~\ref{sec.asympt}
and~\ref{sec.numer}, which notably allow an evaluation of the impact
of the number of bottlenecks on the mean transfer time.
%%%%
%%%
\section{Best-effort networks}\label{sec.best}

Consider the following network model: a set $\L$ of links is given,
where each link $\ell\in\L$ has an associated capacity or bandwidth
$C_\ell>0$. A set $\R$ of routes is given, each route being identified
with a subset of links. Fig.~\ref{fig.linear} illustrates the so-called
linear network: it consists of $L$ links with equal capacity, route $0$
which crosses each link, and routes $r=1,\ldots,L$ which cross a
single link. 

To each route $r$ are associated two parameters: $\lambda_r$ is the
arrival rate of new transfer requests along route $r$, and $\sigma_r$
is the average document size. We make the following standard
simplifying assumptions: requests for document transfers along route
$r$ arrive at the instants of a Poisson process with intensity
$\lambda_r$, while the corresponding document sizes are mutually
independent, independent of the arrival times, and drawn from an
exponential distribution of mean $\sigma_r$.

These traffic assumptions make the process specifying the number of
transfers in progress on different routes Markovian (see below) and
thus greatly simplify analysis. The Poisson arrivals assumption is not
unreasonable in a large network. In view of the insensitivity of
performance results for an isolated link to the exact document size
distribution, we do not expect divergence of the real distribution
from the exponential size assumption to invalidate the derived
conclusions. However, the main reason for assuming an exponential
distribution is clearly one of analytical tractability.

% %%
% \begin{figure}
% \begin{center}
% \setlength{\unitlength}{0.0055in}
% %
% \begin{picture}(535,86)(0,-10)
% \texture{c0c0c0c0 0 0 0 0 0 0 0 
% 	c0c0c0c0 0 0 0 0 0 0 0 
% 	c0c0c0c0 0 0 0 0 0 0 0 
% 	c0c0c0c0 0 0 0 0 0 0 0 }
% \shade\path(155,25)(155,45)(55,45)
% 	(55,25)(155,25)
% \path(155,25)(155,45)(55,45)
% 	(55,25)(155,25)
% \shade\path(295,25)(295,45)(195,45)
% 	(195,25)(295,25)
% \path(295,25)(295,45)(195,45)
% 	(195,25)(295,25)
% \shade\path(515,25)(515,45)(415,45)
% 	(415,25)(515,25)
% \path(515,25)(515,45)(415,45)
% 	(415,25)(515,25)
% \path(45,15)(45,30)(165,30)(165,15)
% \path(163.000,23.000)(165.000,15.000)(167.000,23.000)
% \path(185,15)(185,30)(305,30)(305,15)
% \path(303.000,23.000)(305.000,15.000)(307.000,23.000)
% \path(405,15)(405,30)(525,30)(525,15)
% \path(523.000,23.000)(525.000,15.000)(527.000,23.000)
% \path(35,40)(305,40)
% \dashline{4.000}(305,40)(405,40)
% \path(405,40)(535,40)
% \path(527.000,38.000)(535.000,40.000)(527.000,42.000)
% \put(100,55){\makebox(0,0)[lb]{\smash{{{\footnotesize link 1}}}}}
% \put(235,55){\makebox(0,0)[lb]{\smash{{{\footnotesize link 2}}}}}
% \put(455,55){\makebox(0,0)[lb]{\smash{{{\footnotesize link L}}}}}
% \put(0,50){\makebox(0,0)[lb]{\smash{{{\footnotesize route 0}}}}}
% \put(35,0){\makebox(0,0)[lb]{\smash{{{\footnotesize route 1}}}}}
% \put(185,0){\makebox(0,0)[lb]{\smash{{{\footnotesize route 2}}}}}
% \put(405,0){\makebox(0,0)[lb]{\smash{{{\footnotesize route L}}}}}
% \end{picture}

% \caption{The linear network}\label{figure1}
% \end{center}
% \end{figure}
%%
\begin{figure}
\begin{center}
\includegraphics[width=\columnwidth]{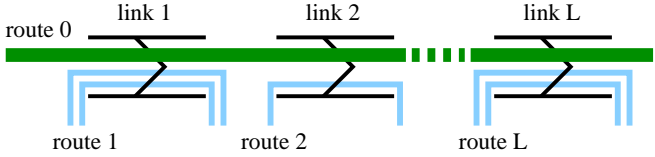}
\end{center}
\caption{The linear network}\label{fig.linear}
\end{figure}
The network state is summarized by the variables $X\egaldef\{x_r,
r\in\R\}$, where $x_r$ denotes the number of transfers in progress
along route $r$. It remains to specify at what speed documents are
transmitted in any given state $X$ in order to turn $X$ into a Markov
process with well defined dynamics. Indeed, given the rate
$\zeta_r(X)$ at which documents along route $r$ are transferred when
the network state is $X$, $X$ is a Markov process with non-zero
transition rates given by
$$
\left\{
\begin{array}{ll}
x_r\rightarrow x_r+1:&\mbox{ rate }\lambda_r,\\
x_r\rightarrow x_r-1:&\mbox{ rate }x_r\zeta_r(X)/\sigma_r.
\end{array}
\right.
$$
A natural assumption would be to consider that each document
receives its fair share of bandwidth. For instance, if as
in~\cite{BerGal:1} fairness is understood as max-min fairness, each
transfer along route $r$ receives a bandwidth share $\zeta\Rfair_r$, where
\begin{equation}\label{eq.capa}
\sum_{r\ni \ell} x_r \zeta\Rfair_r\leq C_\ell,\ \forall\ell\in\L,
\end{equation}
and for every route $r$, there is at least one link $\ell\in r$ such that
\begin{equation}\label{eq.maxmin}
 \sum_{r'\ni \ell}x_{r'}\zeta\Rfair_{r'}=C_\ell, 
 \mbox{ and }\zeta\Rfair_r=\max_{r'\ni \ell} \zeta\Rfair_{r'}.
\end{equation}
These two conditions uniquely determine the bandwidth shares $\zeta\Rfair_r$.
Having specified the Markov process $X$, one can then
attempt to study its steady state properties, identifying  the
conditions on the load parameters 
\[
\rho_\ell\egaldef\frac{1}{C_\ell}\sum_{r\ni\ell}\lambda_r\sigma_r
\] 
under which it is ergodic and, when it is, determining the stationary
distribution. Mean transfer times $T_r$ along each route $r$ can then
be computed using Little's law: $T_r=\EE x_r/\lambda_r$.

It turns out that explicit formulas for steady state distributions are
typically beyond reach. A notable exception is the linear network,
with bandwidth shares being allocated to realize proportional rather
than max-min fairness; see~\cite{mass2}. In order to obtain formulas
in other cases, one therefore has to resort to asymptotics on various
parameters. For instance, for the linear network with max-min fair
rate sharing, the regime where the arrival rate $\lambda_0$ along
route $0$ goes to zero (essentially, a form of light traffic analysis)
is considered in~\cite{mass2}; this leads to approximate formulas for
$T_0$. It can be shown, in particular, that $T_0$ increases as the
logarithm of the number of links $L$ when $L$ increases. This is in
contrast to the case of proportionally fair sharing where it increases
linearly in $L$.

The main purpose of this paper is to investigate an alternative
asymptotic regime where it is the network topology which evolves. The
precise description of this limiting regime will be given in
Section~\ref{sec.asympt}.

In the following sections, we consider
bandwidth allocations according to the following ``min'' policy: given
the network state $X$, each transfer along route $r$ receives a
bandwidth share $\zeta\Rmin_r$ given by
\begin{equation}
\zeta\Rmin_r\egaldef\min_{\ell\in r}\frac{C_\ell}{X_\ell},
\end{equation}
where we have introduced the notation $X_\ell=\sum_{r'\ni \ell}x_{r'}$ to
represent the total number of transfers making use of link~$\ell$.

It is easy to check that this policy satisfies the capacity
constraints~(\ref{eq.capa}). Moreover, it is sub-optimal with respect to
the max-min fairness policy, as shown in the next theorem.

\begin{theo}\label{theo.order}
Under the same initial conditions, the vector $X\Rfair(t)$ for the
system under the $\zeta\Rfair$ allocation policy is stochastically
smaller than $X\Rmin(t)$, corresponding to the $\zeta\Rmin$ allocation.
\end{theo}

\begin{proof}
Assume that, for some $t$, $X\Rfair(t)\leq X\Rmin(t)$. Then, with the
notation of~(\ref{eq.maxmin}),
\begin{eqnarray*}
\zeta\Rfair_r 
   &=& \max_{r'\ni \ell} \zeta\Rfair_{r'}\\
   &\geq& \frac{1}{X\Rfair_{\ell}(t)}
                  \sum_{r'\ni \ell} x\Rfair_{r'}(t)\zeta\Rfair_{r'}
        \,=\, \frac{C_\ell}{X\Rfair_\ell(t)}\\
   &\geq& \frac{C_\ell}{X\Rmin_\ell(t)}
        \,\geq\, \zeta\Rmin_r.
\end{eqnarray*}

Thus, using a coupling argument, one can define the processes
$X\Rfair$ and $X\Rmin(t)$ in such a way that $X\Rfair(t)\leq
X\Rmin(t)$ for all $t>0$.
\end{proof}

The previous theorem motivates the study of the min policy, as it
implies for instance that mean transfer times $T_r$ under the min
policy provide upper bounds on the corresponding transfer times under
the max-min policy.

%% file: inf-ergo.tex
\newcommand{\D}{\mathop{\Delta}}
\newcommand{\hx}{\hat x}
\newcommand{\HX}{\hat X}

\section{Ergodicity conditions}\label{sec.ergo}

In the following we demonstrate that min and max-min bandwidth sharing
policies have a stationary regime under the usual conditions, i.e.\
when the load on each link $\ell$ is less than~1:

\begin{theo}
Under the allocation policies $\zeta\Rfair$ and $\zeta\Rmin$, the
network is
\begin{enumerate}
\item[(\emph{i})] ergodic if $\max_{\ell\in\L} \rho_\ell < 1$;
\item[(\emph{ii})] transient if $\max_{\ell\in\L}\rho_\ell > 1$.
\end{enumerate}
\end{theo}

This result has already been proven for the max-min policy in
\cite{konst}; we note that, by Theorem~\ref{theo.order}, ergodicity
under the min policy implies ergodicity under the max-min policy, thus
the treatment of the min policy given below provides an alternative
proof to that of~\cite{konst}. However we feel that, since the proof
below is simpler and uses only elementary Lyapunov functions results,
it should be easier to adapt to a more complicated situation.
Transience under condition~(\emph{ii}) is in fact valid for any
allocation policy which meets the capacity
constraints~(\ref{eq.capa}).

\begin{proof}
Consider the discrete time chain $(\HX(n), n\in\NN)$ describing the
sequence of states visited by the continuous time jump process $X$.
Transitions from a given state $\HX=\big(\hx_r, r\in \R\big)$ satisfy
\begin{eqnarray*}
\PP\bigl[\D \hx_r(n)=1 \bigm| \HX(n)=\HX\bigr] 
  &=&\frac{\lambda_r}{D}\,,\\[0.3pt] 
\PP\bigl[\D \hx_r(n)=-1 \bigm| \HX(n)=\HX\bigr] 
  &=& \frac{1}{D}\,\frac{\hx_r}{\sigma_r}\min_{\ell\in r} \frac{C_{\ell}}{\HX_\ell} ,
\end{eqnarray*}
where $\D \hx_r(n)\egaldef \hx_r(n+1)-\hx_r(n)$ and
\begin{eqnarray*}
D &\egaldef&\sum_{r\in \R}\Bigl(\lambda_r
               +\frac{\hx_r}{\sigma_r}\min_{\ell\in r} \frac{C_{\ell}}{\HX_{\ell}}\Bigr)\\
  &\leq& |\R|\cdot\Bigl(\max_{r\in \R}\lambda_r
                      +\max_{r\in \R}\frac{1}{\sigma_r}
                       \cdot\max_{\ell\in\L}C_\ell\Bigr)\\
  &\egaldef&D'.
\end{eqnarray*}

Ergodicity of the continuous time process $X$ will follow from that of
$\HX$ and from the fact that the mean sojourn times in each state $X$
are bounded from above uniformly in $X$ (or equivalently, that the
jump rates out of each state $X$ are bounded away from zero), a
property which is easily verified.

\medskip
\noindent\emph{Sufficient condition.} 
Assume $\rho_M\egaldef\max_{\ell\in\L} \rho_\ell < 1$, and define the
Lyapunov function
\[
f(\HX) \egaldef \sum_{r\in\R}\sum_{1\leq k\leq \hx_r}\gamma_r^k,
\]
where $\gamma_r>1$ will be chosen later. The structure of the
function, which may seem unnatural, has been chosen for the sake of
computation; it is in fact of the general form
$\sum_{r\in\R}\beta_r\gamma_r^{\hx_r}+K$, for appropriate constants
$\beta_r$ and $K$.

In order to express the transition rates in terms of $\hx_r$ and $\rho_\ell$,
remark that
\[
 \HX_\ell = \sum_{r\ni \ell} \hx_r 
     = \sum_{r\ni \ell} \lambda_r \sigma_r \frac{\hx_r}{\lambda_r \sigma_r}
     \leq \rho_\ell C_\ell \max_{r\ni \ell}\frac{\hx_r}{\lambda_r \sigma_r}.
\]

Then, using the notation
\[
  \hx^*_r\egaldef \frac{\hx_r}{\lambda_r \sigma_r},\quad
  \hx^*_M\egaldef \max_{r\in\R} \hx^*_r,
\]
we have, $\forall r\in\R$,
\[
 \PP\bigl[\D \hx_r(n)=-1 \bigm| \HX(n)=\HX\bigr] 
   \geq \frac{\lambda_r}{D\rho_M}\frac{\hx^*_r}{\hx^*_M}.
\]

Thus,
\begin{eqnarray*}
\lefteqn{\EE\bigl[f(\HX(n+1))-f(\HX(n)) \bigm| \HX(n)=\HX\bigr]}\qquad\\[0.3cm]
&=& \sum_{r\in\R}\Big(\gamma_r^{\hx_r+1}
 \PP\bigl[\D \hx_r(n)=1 \bigm|  \HX(n)=\HX\bigr] \\
& &\hphantom{\sum_{r\in\R}\Big(}- \gamma^{\hx_r}
 \PP\bigl[\D \hx_r(n)=-1 \bigm|  \HX(n)=\HX\bigr]\Big) \\[0.3cm]
&\leq& \sum_{r\in\R}\frac{\lambda_r}{\rho_M D} \gamma_r^{\hx_r}
 \Bigl[\rho_M\gamma_r - \frac{\hx^*_r}{\hx^*_M}\Bigr]
\end{eqnarray*}

Let $\gamma_r\egaldef\gamma^{\frac{1}{\lambda_r \sigma_r}}$, where $\gamma$
is such that
\[
 \rho_M\gamma_r=\rho_M\gamma^{\frac{1}{\lambda_r \sigma_r}}<\theta<1,
 \ r\in\R,
\]
for some real number $\theta$ satisfying $\rho_M<\theta<1$. 
The following inequality sums up what we have so far:
\begin{eqnarray*}
\lefteqn{\EE\bigl[f(\HX(n+1))-f(\HX(n)) \bigm| \HX(n)=\HX\bigr]}\hspace{3cm}\\
  &\leq& \sum_{r\in\R}\frac{\lambda_r\gamma^{\hx^*_r}}{\rho_M D} 
       \Bigl[\theta - \frac{\hx^*_r}{\hx^*_M}\Bigr].
\end{eqnarray*}

Let $\alpha$ be a real number such that $\theta<\alpha<1$. The
following quantities will be evaluated separately
\[
\left\{ \begin{array}{ccl}
\Sigma_1 &\egaldef&
 \displaystyle \sum_{r: \hx^*_r > \alpha \hx^*_M}
       \frac{\lambda_r\gamma^{\hx^*_r}}{\rho_M D} 
       \Bigl[\theta - \frac{\hx^*_r}{\hx^*_M}\Bigr] ,\\[0.3cm]
\Sigma_2 &\egaldef&
 \displaystyle \sum_{r: \hx^*_r \leq \alpha \hx^*_M}
       \frac{\lambda_r\gamma^{\hx^*_r}}{\rho_M D} 
       \Bigl[\theta - \frac{\hx^*_r}{\hx^*_M}\Bigr].
\end{array} \right.
\]

Since $\Sigma_1$ is a sum of negative terms, the following bound holds,
for any $r_0$ is such that $\hx^*_{r_0}=\hx^*_M$,
\[
\Sigma_1 \leq \frac{\lambda_{r_0}\gamma^{\hx^*_M}}{\rho_M D}(\theta - \alpha)
         \,\leq\, \frac{\gamma^{\hx^*_M}}{\rho_M D}(\theta - \alpha)
                  \min_{r\in \R}\lambda_r \,<\,0.
\]

Bounding $\Sigma_2$ is straightforward: 
\[
\Sigma_2 \leq \sum_{r: \hx^*_r \leq \alpha \hx^*_M}
       \frac{\lambda_r\gamma^{\alpha \hx^*_M}}{\rho_M D}\theta 
  \,\leq\, \frac{\gamma^{\hx^*_M}}{\rho_M D}\gamma^{(\alpha-1)\hx^*_M}|\R|\theta 
      \max_{r\in\R}\lambda_r.
\]

Now, if $C>0$ and $\epsilon>0$ are chosen to satisfy the inequality
\[
  (\theta - \alpha) \min_{r\in \R}\lambda_r 
  +\gamma^{(\alpha-1)C}|\R|\theta \max_{r\in\R}\lambda_r
  \,\leq\,-\epsilon,
\]
we have, $\forall \HX\in\{\hx^*_M>C\}$,
\begin{eqnarray*}
\lefteqn{\EE\bigl[f(\HX(n+1))-f(\HX(n))\bigm|\HX(n)=\HX\bigr]}\qquad\qquad\\
 &=& \Sigma_1 + \Sigma_2
 \,\leq\, -\epsilon\frac{\gamma^{\hx^*_M}}{\rho_M D}
  \,\leq\, -\epsilon\frac{\gamma^{C}}{\rho_M D'} \,<\, 0.
\end{eqnarray*}

Since $\{\hx^*_M\leq C\}$ is a compact set, Foster's theorem applies
(see e.g.~\cite{FMM}) and the Markov chain is ergodic.

\medskip
\noindent\emph{Necessary condition.}
Assume now that there exists $\ell_0$ such that $\rho_{\ell_0} > 1$. Defining
\[
  g(\HX) \egaldef \sum_{r\ni \ell_0}\sigma_r \hx_r,
\] 
we immediately have,
\begin{eqnarray*}
\lefteqn{\EE\bigl[g(\HX(n+1)) - g(\HX(n)) \bigm| \HX(n)=\HX\bigr]}\qquad\qquad\\
  &=&\sum_{r\ni \ell_0}\sigma_r\Bigl(\PP\bigl[\D \hx_r(n)=1 \bigm|  \HX(n)=\HX\bigr]\\
  & & \hphantom{\sum_{r\ni \ell_0}\sigma_r\Bigl[}
      - \PP\bigl[\D \hx_r(n)=-1 \bigm|  \HX(n)=\HX\bigr]\Bigr)\\
  &\geq& \frac{1}{D'}[C_{\ell_0}\rho_{\ell_0}-C_{\ell_0}] \,>\, 0.
\end{eqnarray*}

Since the jumps are bounded, the chain is transient.
\end{proof}

%% file: inf-asympt.tex
\section{Mean field analysis of large networks}\label{sec.asympt}

It does not appear possible to obtain closed form expressions for the
stationary distribution of the best-effort network state under the min
policy. We therefore turn to the study of these stationary
distributions under a limiting regime on network size and topology. A
similar approach has previously been successfully applied to loss
networks (see~\cite{kelly,graham}, and references therein), and to
queueing networks in~\cite{VveDobKar:1,DelFay:2}. It is inspired by
the so-called mean field models of statistical physics.

\begin{figure}
\begin{center}
\includegraphics[width=0.6\columnwidth]{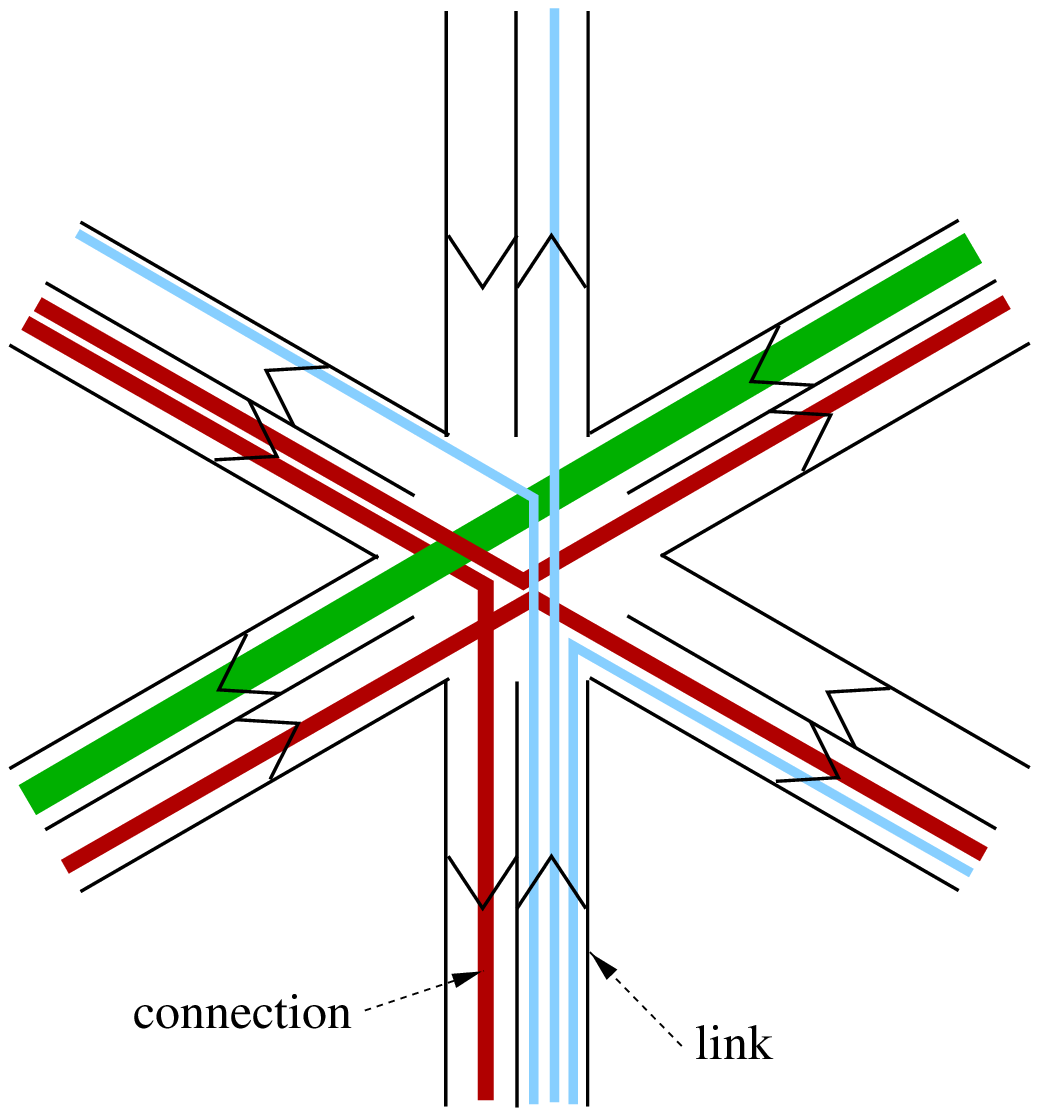}
\end{center}
\caption{A star-shaped network}
\label{fig.etoile}
\end{figure}
Mean field analysis in the present context is best illustrated by the
star-shaped network of Fig.~\ref{fig.etoile}. This network has $N/2$
branches, each consisting of one inbound and one outbound link (thus
implicitly $N$ is an even number), and all links have unit capacity.
Each route $r$ connects the endpoints of two branches via the center
node. It has an associated arrival rate $2\lambda/N$ and mean
message length $\sigma$. The factor $2/N$ is introduced to make
the total load on each link, $\rho=\lambda \sigma$, independent of the
number of links $N$. As discussed below, when $N$ goes to infinity
the number of ongoing transfers on any link becomes independent of the
number of ongoing transfers on any finite collection of the other
links (this was termed the ``chaos propagation'' property
in~\cite{graham}). This allows us to derive fixed point equations for
the probability distribution of the number of transfers in progress on
any link.

\subsection{Symmetrical star-shaped networks}\label{sec.sym}

Although amenability to a mean field analysis is a significant
motivation for considering the star-shaped topology, it should be
noted that it is also relevant to the study of real networks. Any
overprovisioned links in a real network are largely transparent to the
throughput of elastic best effort flows. Only bottleneck resources,
typically located in the access network and within Web servers, need
to be included in the network model. The star-shaped network may thus
be considered to represent any network with a well provisioned
backbone where throughput is limited by bottlenecks at the source and
destination edges. For example, inbound links might represent Web
server CPU, while outbound links correspond to the last hop of an
ISP's interconnection network. This discussion not only motivates the
consideration of such a topology, but also suggests that letting $N$
go to infinity might indeed be realistic if $N$ represents the number
of Web servers over the Internet. Of course, there would be no reason
in practice to assume symmetry. This assumption is introduced solely
for reasons of tractability.

Although our focus is on the star-shaped topology, the mean field
approach can be applied to other symmetrical topologies. It thus
allows one to consider routes with more than two hops. The
corresponding extended model is described in detail in
Section~\ref{sec.symmetrical} below, where the corresponding fixed
point equations are derived. Section~\ref{sec.anal} then presents
analytical results for the star-shaped network. Results of numerical
investigation of the fixed point equations are reported in
Section~\ref{sec.numer}.

\subsection{Fixed point equations for large symmetrical networks}
\label{sec.symmetrical}

We use the following notation in the sequel.
\begin{itemize}
\item $N$: total number of links;
\item $L$: length of a route through the network;
\item $R\N$: number of routes going through a given link;
\item $X\N_\ell$: number of active connections on link
$\ell\in\L$, in stationary state; 
\item $x\N_{r}$: number of active connections on route
$r$, going through links $r(1),\ldots,r(L)$;
\item $\lambda$: arrival rate on a link;
\item $\sigma$: mean message length;
\item $\rho\egaldef \lambda \sigma$: load of a link.
\end{itemize}

We have implicitly assumed here that the number of routes going
through a link, $R\N$, is the same for all links. We shall in fact
assume further that the network topology is the same, as seen from any
route. We do not attempt to give a formal definition of this symmetry
assumption here. The reader is referred to~\cite{graham} for a
thorough discussion on the minimal symmetry assumptions required.
Symmetry implies notably that each route has the same number of hops
and the same traffic parameters. The star-shaped network discussed
above constitutes an example of such a symmetrical network when $L=2$
(with $R\N=N/2$).

\begin{figure}
\begin{center}
\includegraphics[width=0.7\columnwidth]{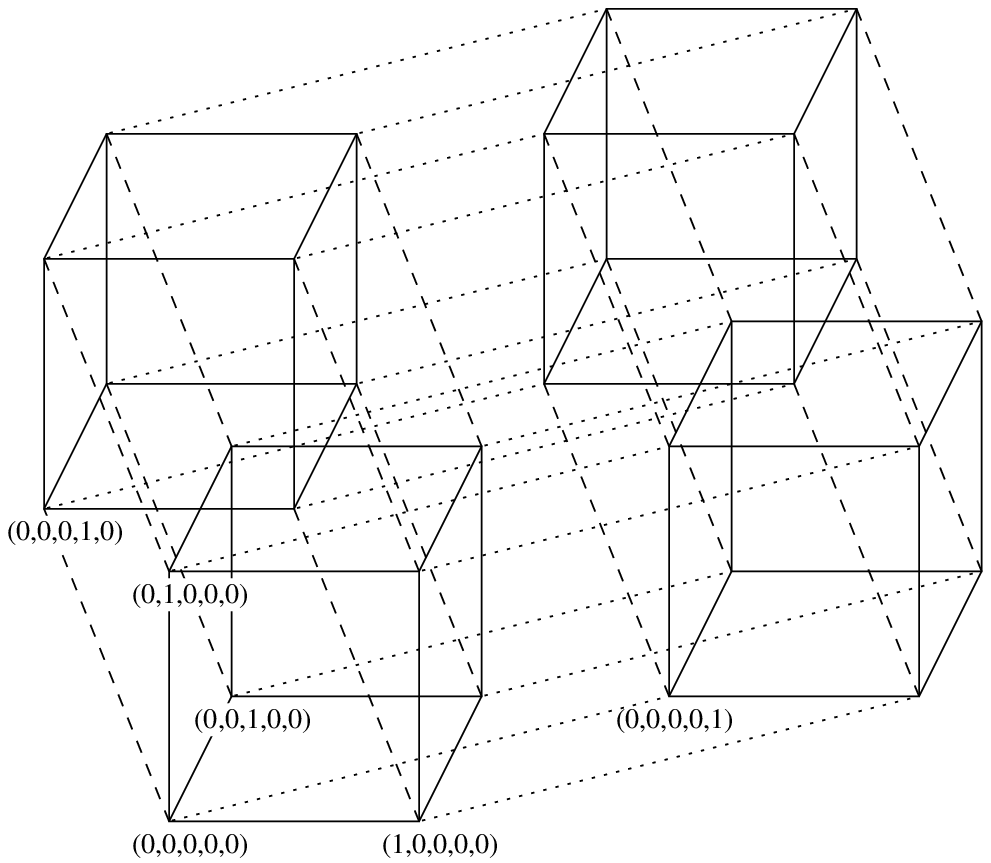}
\end{center}
\caption{A hypercube-shaped network with dimension $5$.}
\label{fig.hypercube}
\end{figure}
It is more difficult to come up with meaningful examples of symmetrical
networks supporting routes with $L>2$: in particular, the network
should not be fully connected, since routes longer than one hop would
then be pointless. One reasonable model is a hypercube
(Fig.~\ref{fig.hypercube}) of large dimension, in which each edge
contains two one-way links.

The hypercube is a classical structure with many symmetries. It is
characterized by its dimension $d$. Its vertices are represented by
$d$-tuples of $0$s and $1$s (e.g.,\ $(0,1,1,0,1)$) and its edges
connect two vertices differing in only one coordinate.

The total number of links in such a network is $N=d2^d$. The number
of routes going through any link is
\[
 R\N=L\frac{(d-1)!}{(d-L)!},
\]
where the only routes considered are the shortest paths between two
vertices which differ in exactly $L$ coordinates. Note that the
results below do not depend on the precise topology of the network.

We now derive the fixed point equations. It should be stressed that
this derivation is heuristic. We clearly mention which steps need
further justification in the course of the derivation. We do believe
that the equations are very good approximations, however, especially
in view of the numerical and simulation results presented in the
following section.

%----------------------------------------------------------------------
%\subsection{Recurrence equations}
%----------------------------------------------------------------------

Assume now that $\rho<1$ and that the system is in stationary state
$X\N=(x\N_r,r\in\R)$. For any $k\geq0$, the proportion of links in
state $k$ is
\[
 \alpha\N_k\egaldef\frac{1}{N}\sum_{\ell\in\L}\ind{X\N_\ell = k}.
\]

By symmetry, it holds that
\[
  \PP(X\N_\ell=k)=\EE\alpha\N_k,\ \forall\ell\in\L.
\]
The chaos propagation assumption\footnote{We have not proven that this
assumption holds. It seems, however, that the techniques developed
in~\cite{graham} could be applied to prove that this is the case,
provided the parameter $R\N$ goes to infinity with $N$.} implies
that $\alpha\N_k$ obeys a law of large numbers:
\[
 \alpha_k \egaldef \lim_{N\to\infty}\alpha\N_k 
 = \lim_{N\to\infty}\PP(X\N_\ell=k) \egaldef \PP(X=k).  
\]

It appears that the dynamics of the system are driven by $\alpha\N_k$,
traditionally referred to as the \emph{mean field}. The following
notation will also be useful:
\[
  \EX\N \egaldef \sum_{k>0}k\alpha\N_k,\qquad
  \EX   \egaldef \EE X_\ell = \sum_{k>0}k\alpha_k.
\]

In order to derive the equation satisfied by the limit stationary
distribution $\alpha_k$, we must first describe the possible
transitions for $\alpha\N_k$. The two cases of interest are
\begin{itemize}
\item arrival on a link with $k$ connections:
\[ 
 \alpha\N_k\to\alpha\N_k-\frac{1}{N},\quad 
 \alpha\N_{k+1}\to\alpha\N_{k+1}+\frac{1}{N}
\]
\item departure from a link with $k>0$ connections:
\[ 
 \alpha\N_k\to\alpha\N_k-\frac{1}{N},\quad 
 \alpha\N_{k-1}\to\alpha\N_{k-1}+\frac{1}{N}
\]
\end{itemize}

The transition corresponding to a new connection arrival has rate
$\lambda$. The main problem is to compute the departure rate from a
link $\ell$, given that it has $X\N_{\ell}=k$ ongoing transfers.
This can be written as
\begin{equation}\label{eq.outrate}
\frac{1}{\sigma}\sum_{r\ni \ell}\EE\Bigl[x\N_r
\min_{\ell'\in r} \frac{1}{X\N_{r(\ell')}} \Bigm| X\N_{\ell}=k\Bigr].
\end{equation}

Since the total number of routes is much larger than the size $N$ of
the network, we assume that the probability of having more than one
connection on a route $r$ is negligible\footnote{This fact is easy to
prove in finite time, but requires more work for the stationary
regime.}, and that the links on route $r$ are independent, conditioned
on $\{x\N_r=1\}$\footnote{This is the point where the heuristic is not
completely exact; it is however likely to be true when $\rho$ tends
either to $0$ or $1$.}. The first property allows to
rewrite~(\ref{eq.outrate}) as
\begin{equation}\label{eq.outrate2}
\frac{1}{\sigma}\sum_{r\ni \ell}\EE[x\N_r\mid X\N_\ell=k]
\EE\Bigl[\min_{\ell'\in r} \frac{1}{X\N_{r(\ell')}} \Bigm| x\N_r=1\Bigr].
\end{equation}

Let $j$ be a given link and let $r$ be one route using link~$j$, i.e.,
$r\ni j$. The distribution of $X\N_j$, conditioned on there being one
connection on $r$, is then
\[
 \PP\bigl[X\N_j=k'\bigm|x\N_r=1\bigr]
    = \frac{\EE\bigl[\ind{X\N_j=k'}\ind{x\N_r=1}\bigr]}{\EE\ind{x\N_r=1}}.
\]

By symmetry, it is possible to sum both sides of the above fraction
over all the routes going through link $j$,
\begin{eqnarray*}
 \PP\bigl[X\N_j=k'\bigm|x\N_r=1\bigr]
    &=& \frac{\EE\bigl[X\N_j\ind{X\N_j=k'}\bigr]}{\EE X\N_j}\\
    &=& \frac{k' \EE\alpha\N_{k'}}{\EE\EX\N}.
\end{eqnarray*}

Departure rate~(\ref{eq.outrate2}) then becomes, in view of the assumed
independence property between the $X_{r(i)}$ given $x_r=1$,
\[
\frac{k}{\sigma} \sum_{k_2,\ldots, k_L=1}^\infty 
          \biggl(\frac{1}{k}
	         \wedge\frac{1}{k_2}
                 \wedge\cdots\wedge \frac{1}{k_L}\biggr)
          \prod_{i=2}^L \frac{k_i\EE\alpha\N_{k_i}}{\EE\EX\N}.
\]

Taking the limit $N\to\infty$, the invariant measure equations follow.
We have:
\begin{equation}
 -\lambda \alpha_0 +\frac{1}{\sigma} \alpha_1 \frac{u_1}{\EX^{L-1}} = 0
                                          \label{invariante_mf0}
\end{equation}
and
\begin{equation}
 \lambda (\alpha_{k-1}-\alpha_k) 
  + \frac{1}{\sigma} \left(\alpha_{k+1} \frac{u_{k+1}}{\EX^{L-1}} 
         - \alpha_k \frac{u_k}{\EX^{L-1}}\right)  
	= 0\label{invariante_mf}
\end{equation}
for $k\geq 1$, where
\begin{equation}\label{eq.uk_L}
u_k \,\egaldef\, 
     k \sum_{k_2,\ldots k_L=1}^\infty 
          \biggl(\frac{1}{k}
	         \wedge\frac{1}{k_2}
                 \wedge\cdots\wedge \frac{1}{k_L}\biggr)
          \prod_{i=2}^L k_i\alpha_{k_i}.
\end{equation}

Equations~(\ref{invariante_mf0}) and~(\ref{invariante_mf})
can be rewritten in a more concise form as
\begin{equation}\label{invariante_EQ_L}
\alpha_{k+1} u_{k+1} = \rho \EX^{L-1}\alpha_k, \qquad \forall k\geq 0.
\end{equation}

The two sets of equations (\ref{eq.uk_L}) and (\ref{invariante_EQ_L})
together constitute the fixed point equations we require. As noted in
the introduction of this section, these equations do not depend on the
topology of the network. The expression for $u_k$ can be simplified.

Let $Y_2,\ldots,Y_L$ be random variables with
distribution
\[
  \PP(Y_i=y)=\frac{y\alpha_{y}}{\EX},\ 2\leq i\leq L,
\]
and let $Y \egaldef \max(Y_2,\ldots,Y_L)$.
Then~(\ref{eq.uk_L}) reads
\begin{equation}\label{eq.ukmin}
 u_k \,=\, k\EX^{L-1}\,\EE \Bigl[\frac{1}{k}\wedge \frac{1}{Y}\Bigr].
\end{equation}

Straightforward calculations yield
\begin{eqnarray*}
\EE \Bigl[\frac{1}{k}\wedge \frac{1}{Y}\Bigr]
 &=& \sum_{y=1}^\infty \Bigl[\frac{1}{k}\wedge \frac{1}{y}\Bigr]\PP(Y=y)\\
 &=& \sum_{y=k}^\infty \frac{1}{y(y+1)} 
          \PP(Y\leq y).
\end{eqnarray*}

The simplified form for $u_k$ is thus, from the basic properties of
the minimum of independent random variables:
\begin{equation}
\label{simpler}
 u_k \,=\, k\sum_{y=k}^{\infty}\frac{1}{y(y+1)}
                         \Bigl[\sum_{m=0}^{y} m\alpha_m\Bigr]^{L-1}.
\end{equation}

Note that, in the case $L=2$ (i.e., for the star-shaped network), the
original equation (\ref{eq.uk_L}) is perhaps simpler than the
equivalent expression (\ref{simpler}). It yields the following form
for the fixed point equations:
\begin{eqnarray}
\alpha_{k+1} u_{k+1} 
  &=& \rho \EX\alpha_k,\label{invariante_EQ}\\
u_k &=&\sum_{y>0}(k\wedge y)\alpha_{y}. \label{eq.uk}
\end{eqnarray}

\begin{remark}\label{rem.asym}\upshape
When considering the star-shaped network as a model for Web transfers
over the Internet, as suggested in Section~\ref{sec.sym}, inbound
links could be seen as the CPU of Web servers and outbound links as
the last hop between the ISP's backbone and the end customers. It thus
makes sense to relax the symmetry assumption we had made between
inbound and outbound links, as the two types of bottlenecks are of a
different nature. We might thus consider a star-shaped network with
$N\inb$ inbound links, $N\outb$ outbound links, inbound (resp.\
outbound) links having capacity $C\inb$ (resp.\ $C\outb$), see
Fig.~\ref{fig.etoile.asym}. Assume the mean message length $\sigma$ is
the same for each two-hop route, and the link capacities $C\outb$ are
fixed. The arrival rate on each route has the form
$\lambda_r\egaldef\lambda/N\inb$, and the load
$\rho\outb\egaldef\lambda\sigma/C\outb$ is less that $1$. The capacity
$C\inb$ of a ``backbone'' link is chosen to ensure
$\rho\inb\egaldef\lambda\sigma N\outb/C\inb N\inb<1$ is fixed when the
size of the system grows. Then, as $N\inb$ and $N\outb$ increase, with
$N\inb/N\outb$ small, the inbound links have many active connections
and a large capacity, while the outbound links remain in ``normal''
utilization. The same approach as above can then be applied, to yield
the set of fixed point equations
\begin{eqnarray*}
\alpha_{k+1}\inb u_{k+1}\outb 
  &=& \rho\inb \EX\outb\alpha_k\inb,\\
\alpha_{k+1}\outb u_{k+1}\inb 
  &=& \rho\inb \EX\outb\alpha_k\outb,
\end{eqnarray*}
where
\begin{eqnarray*}
u_k\inb 
  &=& \sum_{y>0}\alpha\inb_{y}\min\Bigl[k, \frac{C\outb}{C\inb}y\Bigr] ,\\
u_k\outb 
  &=& \sum_{y>0}\alpha\outb_{y}\min\Bigl[\frac{C\outb}{C\inb}k, y\Bigr] ,
\end{eqnarray*}
and $\alpha_k\inb$ (resp.\ $\alpha_k\outb$) represents the proportion
of inbound (resp.\ outbound) links with $k$ ongoing transfers.
\end{remark}
\begin{figure}
\begin{center}
\includegraphics[width=0.7\columnwidth]{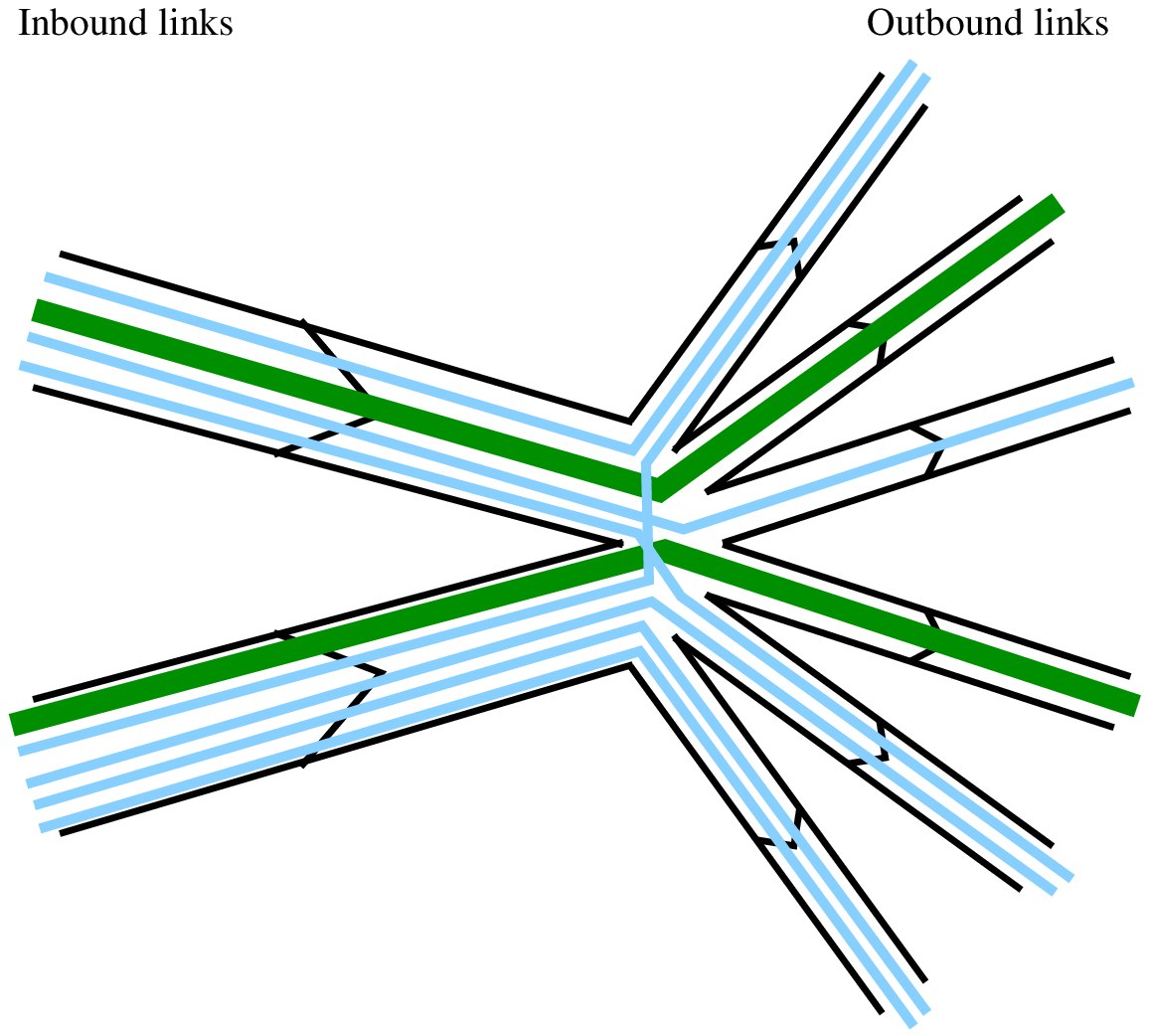}
\end{center}
\caption{The asymmetrical star-shaped network}
\label{fig.etoile.asym}
\end{figure}

\subsection{Analytical results for $L=2$}\label{sec.anal}

While equation~(\ref{invariante_EQ_L}) looks superficially like a
``birth and death process'' equation, it is in fact non-linear due to
the fact that $u_k$ and $\EX$ both depend on the $\alpha_y$.

 From~(\ref{eq.ukmin}), one clearly sees that $u_k$ is increasing in
$k$, and tends to $\EX^{L-1}$ when $k\to\infty$. Therefore, $\alpha_k$
is increasing as long as $u_k<\rho\EX^{L-1}$, and decreasing after
that. This means that the $\alpha_k$ form a modal distribution,
which maximal value is attained at
\[
 k_0 = \max\{k>0,\ u_k<\rho\EX^{L-1}\}.
\]

We now present analytical results on the solution of the fixed point
equations for $L=2$. The proof of these results can be found
in~\cite{FayLas:7}. It relies heavily on functional analysis.

\medskip
Equations (\ref{invariante_EQ})--(\ref{eq.uk}) have an unique solution
for $\rho<1$ and, when $\rho\to1$, the following asymptotic
expansions hold:
\begin{eqnarray*}
 \EE X&\approx&\frac{1}{(1-\rho)^2A},\\
 \lim_{k\to\infty}\rho^{-k}\PP(X=k) 
       &\approx& (1-\rho)B\exp\Bigl[\frac{1}{(1-\rho)A}\Bigr],
\end{eqnarray*}
where $A$ and $B$ are non-negative constants. 
\medskip

Thus, under the min policy, any link in the star-shaped network has a
mean queue length which is one order of magnitude larger than for a
single server queue with the same load ($\rho/(1-\rho)$). Its tail
distribution is still geometrical with factor $\rho$.

It is possible to give an expression for the constant $A$: if $c$ and
$v$ are solutions of the following system of differential equations,
\[
\begin{cases}
zc'(z) + c(z) v(z) =0,\\
zv''(z) + v'(z) = c(z),\\
v(0) = 0, \quad v'(0)=1,\quad c(0)=1,
\end{cases}
\]
then $A$ can be written as follows:
\[
 A=\int_0^\infty c(z)dz=\lim_{z\to\infty}zv'(z)\approx 1.30.
\]

Since this system is numerically highly unstable, it has proven
difficult (with the ``Livermore stiff ODE'' solver from MAPLE) to
derive a better estimate for~$A$.

It is interesting to note that the function $w(y)\egaldef v(e^y)+1$
satisfies the so-called \emph{Blasius}~\cite{Bla:1} equation
\[
 w'''(y)+w(y)w''(y)=0,
\]
which describes a laminar boundary layer along a flat plate (see,
e.g.,~\cite{Sch:1}).

%% file: inf-numer.tex
\section{Numerical analysis and simulations}\label{sec.numer}
%----------------------------------------------------------------------

\begin{figure}
\begin{center}
\includegraphics[width=0.9\columnwidth]{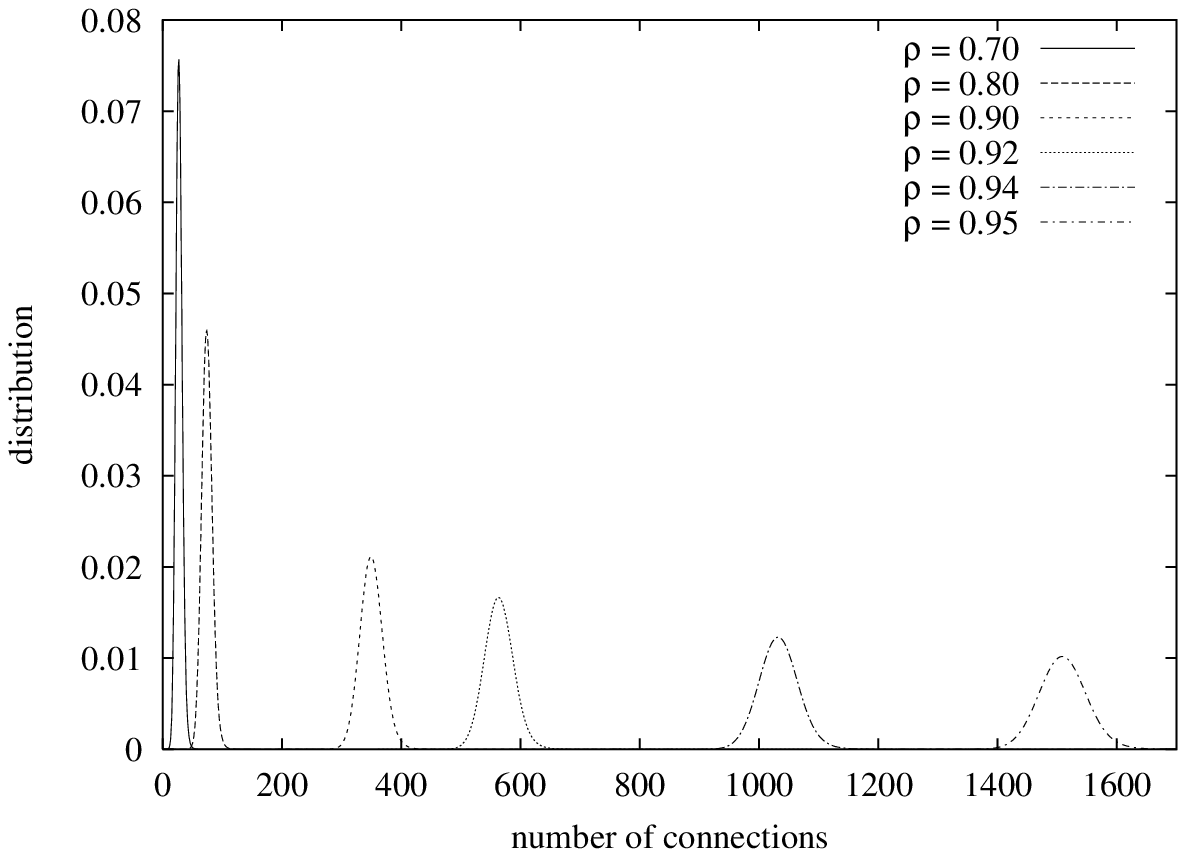}
\end{center}
\caption{Distribution of the mean-field probabilities $\PP(X=k)$ for
$L=20$ and different values of the load $\rho$.}
\label{fig.L=20}
\end{figure}
While the analytical results of Section~\ref{sec.anal} give some good
estimates, they are only valid in the heavy traffic regime $\rho\to
1$. In addition, similar results for $L>2$ are not available. We thus
resort to numerical resolution of the equations to gain a better
understanding of the performance of transfers across large symmetrical
networks. The very form of the equations suggests the use of a
fixed-point method for this numerical resolution: starting from
\emph{a priori} values $(\alpha_k^{\scriptscriptstyle(0)},k\geq0)$,
the algorithm computes the corresponding $u_k^{\scriptscriptstyle(0)}$
from~(\ref{eq.uk_L}), and then new values
$\alpha_k^{\scriptscriptstyle(1)}$ from~(\ref{invariante_EQ_L}).
Provided special care is taken to avoid instabilities, the iteration
of this process converges rapidly (less than $100$ steps). Sample results 
are shown in Fig.~\ref{fig.L=20}, corresponding to a large
symmetrical network with routes of length $L=20$.

As clearly seen in the figure, the distributions are very different
from what would be obtained for routes of length $L=1$. In this case,
the system consists of a collection of independent M/M/$1$ queues and
the associated distribution $\{\alpha_k\}$ is geometric. For $L>1$,
the $\{\alpha_k\}$ distributions are markedly modal and the positions
of the peak values are roughly proportional to $(1-\rho)^{-2}$, a fact
has only been proven in Section~\ref{sec.anal} for the case $L=2$.
Moreover, since the shape of the distribution is rather narrow, this
position roughly coincides with the mean number of active connection
(as can be seen from the raw data).
\begin{figure}
\begin{center}
\includegraphics[width=0.9\columnwidth]{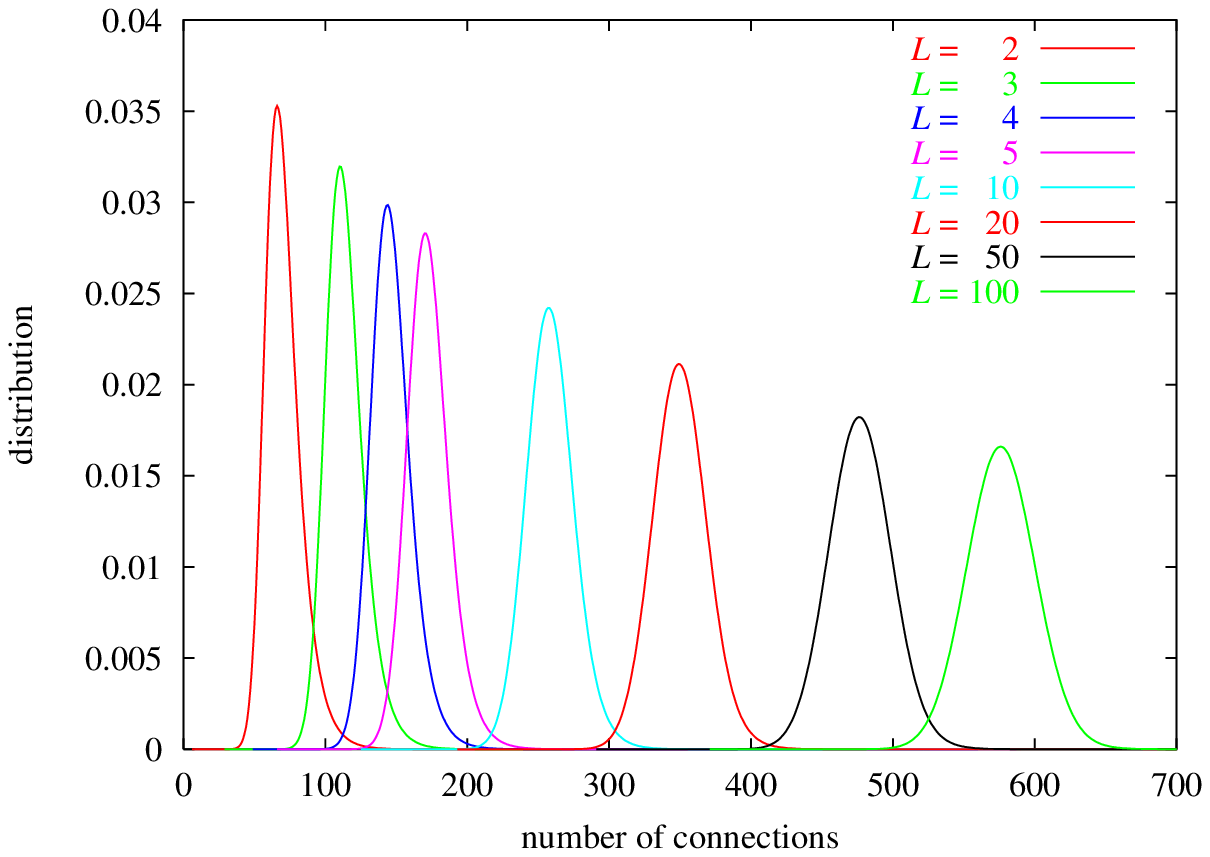}
\end{center}
\caption{Distribution of the mean-field probabilities $\PP(X=k)$ for
different values of the route length $L$ and load $\rho=0.9$.}
\label{fig.rho=.90}
\end{figure}

The impact of route length is illustrated in Fig.~\ref{fig.rho=.90}.
It seems that the mean number of active connections (which is
again approximately the peak value of the distribution) is roughly
proportional to $\log L$. Note that a logarithmic growth rate is very
slow suggesting that, beyond 2 or 3, the number of bottlenecks does not 
have a significant impact on mean transfer times. They depend much more 
on the load $\rho$. 

The results presented so far only concern the solution of the fixed
point equations. As mentioned earlier, there are gaps in the
derivation of these equations. To assess their quality and to
investigate the accuracy of the asymptotic approximation for finite
size networks, we ran a number of simulations of the star-shaped
network. Fig.~\ref{fig.simul.rho=.90} displays the corresponding
results when the load on each link is set to $\rho=0.9$ for a varying
number of links. The agreement between the simulation results and the
fixed point equation results is excellent for $N=100$ links and
improves as $N$ increases.
\begin{figure}
\begin{center}
\includegraphics[width=0.9\columnwidth]{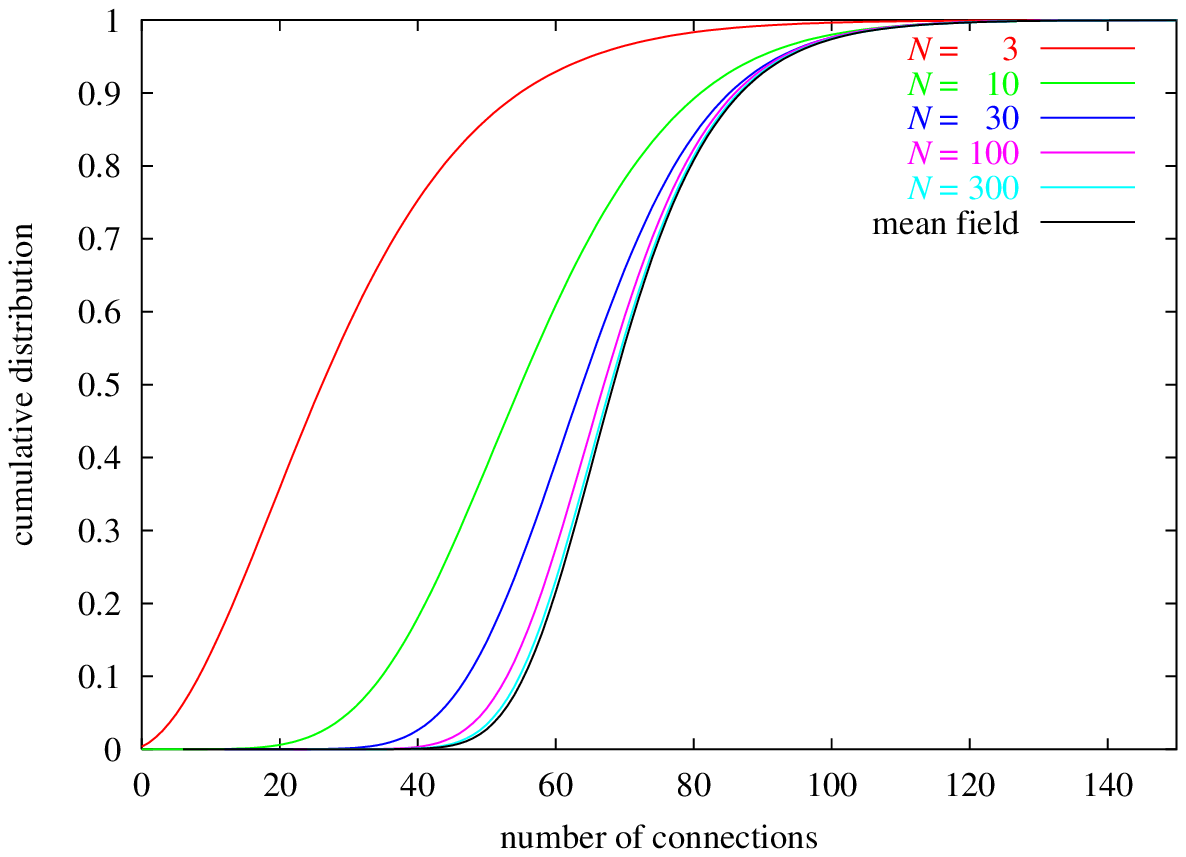}
\end{center}
\caption{Cumulative distribution $\PP(X\leq k)$ for
different values of $N$ (simulation) and infinite size (fixed point)
for a star-shaped network with load $\rho=0.9$.}
\label{fig.simul.rho=.90}
\end{figure}

\section{Conclusions} 

In this paper, we have considered a class of Markov processes called
best-effort networks which constitutes a natural probabilistic model
for evaluating the performance of document transfers over data
networks such as the Internet. Unlike almost all previous work, this
model accounts for the random nature of traffic: document transfers
begin at the epochs of a certain arrival process and the size of each
document is drawn from a given probability distribution. In the
interests of tractability we assumed Poisson arrivals and
exponentially distributed sizes. We introduced the ``min'' bandwidth
sharing policy as a conservative approximation to the more classical
max-min policy. Necessary and sufficient ergodicity conditions for
best effort networks under the min and max-min policies have been
established.

In order to pursue the analysis of the stationary distributions of the
number of transfers in progress, we have resorted to large network
asymptotics applying the mean field approach of statistical physics.
This enabled us to derive fixed point equations for the probability
distribution of the number of ongoing transfers on a given network
link. The validity of these equations has been established by
comparing their solution with the results of simulations.

Analytical and numerical results show how the mean transfer time
depends on the number of bottleneck links and their load. The steady
state distribution in networks where routes have several bottlenecks
($L>1$) has a marked modal behavior. This is significantly different
to the geometric distribution which holds when routes have a single
bottleneck ($L=1$). Performance is also much more sensitive to link
load $\rho$ for multiple bottleneck routes: as $\rho\to 1$, mean
transfer time increases like $1/(1-\rho)^2$ in the case $L=2$, whereas
the dependence is in $1/(1-\rho)$ when $L=1$. Finally, the impact of
the number of hops per route $L$ appears small (given that $L>1$)
compared to that of parameter $\rho$. This suggests that the
star-shaped network is perhaps a sufficiently complex model, and that
the study of symmetrical networks with $L>2$ is less relevant.

The work presented here can be pursued in several directions. On the
theoretical side, the analytical results presented in
Section~\ref{sec.anal} constitute a first step to understanding the
solution of the fixed point equations which could be taken further.
Another challenging theoretical question is to improve the fixed point
equations in a rigorous way. On a more practical side, the fixed point
equations might be simplified so as to find simple approximate
formulas for mean transfer times as a function of key parameters (such
as $\rho^{\mbox{\tiny{in}}}$, $\rho^{\mbox{\tiny{out}}}$ in the case
of the asymmetrical star-shaped network described in
Remark~\ref{rem.asym}). Such approximate formulas could then lead to
engineering rules for capacity planning.

We view the present study as a preliminary investigation into the
performance of best effort networks with multiple bottleneck links. A
significant result of this investigation is the discovery that the
extension of the processor sharing model valid for a single bottleneck
proves to be very hard. There appears to be no simple parallel to the
familiar fixed point techniques used in loss networks. The problem is,
however, of considerable practical importance for providers seeking to
engineer their network to ensure adequate throughput for document
transfers. We hope therefore that this paper will incite further work
and the development of alternative heuristic approaches.